\documentclass{article}

\usepackage{arxiv}

\usepackage[utf8]{inputenc} 
\usepackage[T1]{fontenc}    
\usepackage{hyperref}       
\usepackage{url}            
\usepackage{booktabs}       
\usepackage{amsfonts}       
\usepackage{nicefrac}       
\usepackage{microtype}      
\usepackage{lipsum}
\usepackage{amsmath}

\usepackage{amsfonts,amssymb}
\usepackage{graphicx}
\usepackage{amsthm}
\usepackage{xcolor}

\newcommand{\R}{\mathbb{R}}

\newtheorem{theorem}{Theorem}[section]
\newtheorem{proposition}{Proposition}[section]
\newtheorem{lemma}{Lemma}[section]
\newtheorem{corollary}{Corollary}[section]

\usepackage{cite}

\newcommand{\p}{\partial}
\newcommand{\bb}{\begin{equation}}
\newcommand{\ee}{\end{equation}}
\newcommand{\ba}{\begin{array}}
\newcommand{\ea}{\end{array}}
\newcommand{\f}{\frac}
\usepackage[all]{xy}
\newcommand{\ds}{\displaystyle}

\numberwithin{equation}{section}

\title{A geometrical demonstration for continuation of solutions of the generalised BBM equation}

\author{
 Priscila~Leal~da Silva and Igor Leite Freire\\
  Centro de Matem\'atica, Computa\c{c}\~ao e Cogni\c{c}\~ao\\
 Universidade Federal do ABC\\
 Santo Andr\'e, Brazil \\
  \texttt{priscila.silva@ufabc.edu.br and pri.leal.silva@gmail.com} \\
  \texttt{igor.freire@ufabc.edu.br and igor.leite.freire@gmail.com}
}

\begin{document}
\maketitle

\begin{abstract}
A simple proof that if the generalised BBM equation has a solution vanishing on an open sent of its domain then the solution is necessarily zero is given. In particular, the only compactly supported solution of the equation under consideration is the identically vanishing one.
\end{abstract}

\keywords{BBM equation \and unique continuation of solutions \and compactly supported solutions}

\section{Introduction}

In this paper we consider the Cauchy problem 
\begin{align}\label{1.1}
\left\{\ba{l}
u_t-u_{txx} + \partial_x f(u)=0,\quad t>0\,\,x\in X,\\
\\
u(0,x)=u_0(x),\quad x\in X,
\ea\right.
\end{align}
where $f$ is a differentiable function of one variable and non-identically vanishing. In \eqref{1.1} and throughout this paper, $X$ either denotes $\mathbb{R}$ (real problem) or $\mathbb{S}$ (periodic problem). 
We recall that $\mathbb{S}$ here is the quotient space $\mathbb{R}/\mathbb{Z}$ obtained through the equivalence relation $x\sim y$ if and only if $x-y\in\mathbb{Z}$. Then we can also identify the circle $\mathbb{S}$ with the interval $[0,1)$.

If we take $f(u)=u+u^2/2$ in \eqref{1.1}, we recover the famous Benjamin-Bona-Mahony (BBM) equation \cite{benjamin} and for this reason we refer to the equation in \eqref{1.1} as for the generalised BBM equation.

In \cite[Theorem 4.1]{rosier} it was proved, under some conditions on the function $f$, a unique continuation property of the solutions of \eqref{1.1} for the periodic case. Based on the recent works \cite{linares,igor}, we give a different and geometrical demonstration of the same property for equation \eqref{1.1}, but differently from \cite{rosier} we deal with both real and periodic settings. Our first result is:

\begin{theorem}\label{T1.1}
Let $u_0\in H ^4(X)$ be an initial data to the Cauchy problem \eqref{1.1}, $u$ the corresponding unique solution and $T$ its lifespan. Moreover, assume that $f$ is such that it does not change sign and $f(x)=0$ if and only if $x=0$. If there exists a non-empty set $\Omega\subset[0,T)\times X$ such that $u(t,x) = 0$ for every $(t,x)\in\Omega$, then $u\equiv 0$ in $[0,T)\times X$.
\end{theorem}

Theorem \ref{T1.1} has a simple but strong implication for both cases. Indeed, given an initial data $u_0\in H^4(X)$, if the corresponding solution $u$, associated to a lifespan $T$, is compactly supported, then there exists a non-empty open set $\Omega\subset [0,T)\times X$ such that $u\Big|_{\Omega}=0$. Therefore, Theorem \ref{T1.1} says that $u\equiv 0$ in $[0,T)\times X$, which yields the following result.

\begin{theorem}\label{T1.2}
Given $u_0\in H^4(X)$, let $u$ be the unique solution with lifespan $T$. If $f$ does not change sign, $x=0$ is the only solution of $f(x)=0$ and $u$ is compactly supported, then $u\equiv 0$ in $[0,T)\times X$.
\end{theorem}

{\bf Outline of the manuscript.} For the proof of Theorem \ref{T1.1} in both real and periodic settings, we will make use of two properties: existence and uniqueness of local solutions in Sobolev spaces and conservation of the $H^1$-norm of the solutions of \eqref{1.1}. Its demonstration is also based on some ideas recently proposed in \cite{linares}, as well as the use of the conservation of the Sobolev norm, as pointed out in \cite{igor}. For this reason, we firstly prove that the problem \eqref{1.1} has local solutions in Section \ref{sec2}. Then, in Section \ref{sec3} we find conserved currents for the equation in \eqref{1.1} and, in particular, from one of them we construct a conserved quantity, which implies on the conservation of the Sobolev norm $\|\cdot\|_{H^1}$ of the solutions. Then, in Section \ref{sec4} we prove our main results. In Section \ref{sec5} we present our final comments.

{\bf Notation.} Throughout this paper, $\|\cdot\|_{H^s}$ denotes either the norm in $H^s(\mathbb{R})$ or $H^s(\mathbb{S})$. The same convention is employed for other norms used in this work. 

\section{Local existence and uniqueness of solutions and conserved quantities}\label{sec2}

Here we establish the local existence and uniqueness of solutions of the problem \eqref{1.1}, in both real and periodic cases. For the real case this result can be inferred from \cite[Theorem 1.2]{priscila}. However, for sake of completeness we present here a unified proof for both cases simultaneously. 

We begin noticing that if we define $F(u)=-\p_x\Lambda^{-2} f(u)$, then the problem \eqref{1.1} is equivalent to
\begin{align}\label{2.1}
\left\{\ba{l}
u_t=-F(u),\\
\\
u(0)=u_0,
\ea
\right.
\end{align}
where $\Lambda^{-2}\phi=g\ast\phi$, $\ast$ denotes the convolution operation, while
\begin{align}\label{2.2}
g(x)=\left\{\ba{ll}
\ds{\f{e^{-|x|}}{2}},&\text{if}\,\,X=\R,\\
\\
\ds{\frac{\cosh(x-\lfloor x \rfloor - 1/2)}{2\sinh(1/2)}}, & \text{if}\,\,X=\mathbb{S}.
\ea
\right.
\end{align}
In \eqref{2.2}, $\lfloor x\rfloor$ denotes the greatest integer function. Our first result is given by:

\begin{lemma}\label{lema2.1}
Assume that $f\in C^\infty(\R)$ and $f(x)=0$ if and only if $x=0$. If $u,v\in H^s(X)$, $s>3/2$, then $f(u)\in H^s(X)$ and $\|f(u)-f(v)\|_{H^s}\leq k\|u-v\|_{H^s}$, where $k$ is a constant depending only on $\|u\|_{L^\infty}$ and $\|v\|_{L^\infty}$. In particular, $\|f(u)\|_{H^s}\leq k\|u\|_{H^s}$. 
\end{lemma}

\begin{proof}
See lemmas 1, 2 and 3 in \cite{const-mol}.

\end{proof}

We note that $\p_x\Lambda^{-2}=\p_xg$, where $g$ is given by \eqref{2.2}, and there exists $c>0$ such that both $\|g\|_{L^\infty}$ and $\|\p_x g\|_{L^\infty}$ are bounded from above by $c$. These facts will be used in the next result.

\begin{theorem}\label{teo2.1}
Suppose that $f$, $u$ and $v$ satisfy the conditions in Lemma \ref{lema2.1}, and $F(u)=\p_x\Lambda^2f(u)$. Additionally, assume that $\max\{\|u\|_{L^\infty},\|v\|_{L^\infty}\}<R$, for some positive constant $R$. Then there exist constants $c_1>0$ and $c_2>0$, depending only on $R$, such that $\|F(u)-F(v)\|_{H^s}\leq c_1\|u-v\|_{H^s}$ and $\|F(u)-F(v)\|_{H^{s-1}}\leq c_2\|u-v\|_{H^{s-1}}$. In particular, $\|F(u)\|_{H^s}\leq c_1\|u\|_{H^s}$.
\end{theorem}

\begin{proof}
By Lemma \ref{lema2.1} we know that $f(u)\in H^s(X)$. As a consequence, $F(u)\in H^{s+1}(X)$, which is continuously embedded in both $H^{s}(X)$ and $H^{s-1}(X)$. We shall only prove that $F$ is locally Lipschitz.

From the previous comments about $g$ and $\p_xg$, we have
$$
|F(u)-F(v)|=\left|\int_X\p_x\Lambda^2(f(u)-f(v))\right|dx\leq\int_X|\p_xg||f(u)-f(v)|dx\leq c|u-v|,
$$
Finally, we observe that taking $v=0$, we have $F(v)=0$. This observation proves the last affirmation.
\end{proof}

\begin{theorem}\label{teo2.2}\textsc{(Local Well-Posedness).}
Assume that $f$ satisfies the conditions in Lemma \ref{lema2.1}. Given $u_0\in H^s(X)$, $s>3/2$, then there exists a maximal time of existence $T>0$ and a unique solution $u$ to the Cauchy problem \eqref{2.1}, with $F(u)=\p_x\Lambda^{-2}f(u)$ satisfying the initial condition $u(0,x)=u_0(x)$ such that $u\in C^0\left([0,T);H^s(X)\right)\cap C^1\left([0,T);H^{s-1}(X)\right)$. Moreover, such solution is continuously dependent on the initial data.
\end{theorem}

\begin{proof}
Lemma \ref{lema2.1} and Theorem \ref{teo2.1} show that the problem \eqref{2.1} satisfies the conditions of Kato's Theorem \cite[Theorem 6]{kato}, which concludes the demonstration.
\end{proof}

\section{Conserved currents and conserved quantities}\label{sec3}

We recall that given a differential equation $F=0$, with independent variables $t,x$ and dependent variable $u$, a pair $(C^0,C^1)$, depending on $t,x,u$ and derivatives of $u$, is said to be a {\it conserved current} for the equation if $\p_tC^0+\p_xC^1$ vanishes on the solutions of the equation. The divergence $\p_tC^0+\p_xC^1=0$ is called {\it conservation law} for the equation.

It is well known \cite[page 266]{olverbook} that if $\p_tC^0+\p_xC^1=0$ on the solutions of $F=0$, then there exists a function $Q$, also depending on $t,x$, $u$ and its derivatives, such that
\bb\label{3.1}
\p_tC^0+\p_xC^1=QF.
\ee

The function $Q$ is called characteristic of the conservation law and it can be founded noticing that ${\cal E}_u(QF)=0$, where ${\cal E}_u$ is the Euler-Lagrange operator, see \cite[Theorem 4.7]{olverbook} for further details.

\begin{theorem}\label{teo3.1}
Assume that $Q=Q(u,u_t,u_x,u_{tt},u_{tx},u_{xx})$ is a characteristic of the equation in \eqref{1.1}. Then $Q=c_1+c_2u+c_3(f(u)+u_{tx})$, where $c_1$, $c_2$ and $c_3$ are arbitrary constants.
\end{theorem}

\begin{proof}
Let ${\cal E}_u$ be the Euler-Lagrange operator, see \cite[Def. 4.3]{olverbook} for its definition. From the condition
${\cal E}_u\left[Q\left(u_t-u_{txx}+\p_xf(u)\right)\right]=0$
we obtain a set of equations to be solved for $Q$, whose solution is $Q=c_1+c_2u+c_3(f(u)-u_{tx})$.
\end{proof}

\begin{corollary}\label{cor3.1}
The conserved current associated to the characteristic found in Theorem \ref{teo3.1} is a linear combination of the currents $(u,-u_{tx}+f(u))$, $((u^2+u_x^2)/2,-uu_{tx}+h(u))$ and $(F(u),(u_{tx}^2-u_{t}^2)/2-f(u)u_{tx}+f(u)^2/2)$, where $F'(u)=f(u)$ and $h'(u)=uf'(u)$.
\end{corollary}
\begin{proof}
Let $\mathop{\rm Div}(C^0,C^1):=\p_t C^0+\p_x C^1$. It is enough to notice that
\begin{align*}
    \left[c_1+c_2u+c_3(f(u)+u_{tx})\right]\left(u_t-u_{txx}+\p_xf(u)\right)=c_1 \mathop{\rm Div}\left(u,-u_{tx}-f(u)\right)\\
    \\
    +c_2 \mathop{\rm Div}\left(\f{u^2+u_x^2}{2},-uu_{tx}+h(u)\right)    +c_3 \mathop{\rm Div}\left(F(u),\f{u_{tx}^2-u_t^2}{2}-f(u)u_{tx}+\f{f(u)^2}{2}\right).
\end{align*}
\end{proof}

We recall that if $\p_tC^0+\p_xC^1=0$ is a conservation law of an equation $F=0$, with independent variables $t$ and $x$, then $$
\f{d}{dt}\int_{\Omega}C^0dx=-C^1\Big|_{\p\Omega},
$$
where $\Omega$ is the set where the variable $x$ varies, with boundary $\p\Omega$ (in case the set is unbounded, we assume $\pm\infty$ as boundaries). If $C^1\big|_{\p\Omega}=0$, then the quantity
\bb\label{3.2}
{\cal H}(t):=\int_{\Omega}C^0dx
\ee
is a constant because ${\cal H}'(t)=0$. The functions $C^0$ and $C^1$ are, respectively, called conserved density and conserved flux, while \eqref{3.2} is known as conserved quantity. For further details and discussions about conservation laws, see \cite{olver-1,olver-2,olver-3}.

We recall that the conserved currents for the BBM equation (eventually replacing $u\mapsto u-1$) were found in \cite{benjamin}, whereas Olver \cite{olver-1} showed that they are the unique conserved currents for the BBM equation. We conjecture that the currents in Corollary \ref{cor3.1} are the unique conserved currents for the equation in \eqref{1.1}.
\begin{theorem}\label{teo3.2}
Assume that $u_0\in H^s(X)$ is an initial data of the Cauchy problem \eqref{1.1}. Then the Sobolev norm $\|u\|_{H^1}$ is conserved.
\end{theorem}
\begin{proof}
Let
\begin{align*}
{\cal H}(t)=\f{1}{2}\int_{X}\left[u(t,x)^2+u_x(t,x)^2\right]dx.
\end{align*}
Since ${\cal H}'(t)=0$, then ${\cal H}(t)=const$. Such a constant can be determined using the initial data, that is,
$$
{\cal H}(0)=\f{1}{2}\int_{X}\left[u(0,x)^2+u_x(0,x)^2\right]dx=\f{1}{2}\int_{X}\left[u_0(x)^2+(u_0(x)')^2\right]dx.
$$
The result is then a consequence of the fact that ${\cal H}(t)=\|u\|_{H^1}^2/2$.
\end{proof}

\begin{corollary}\label{cor3.2}
Given $u_0\in H^s(X)$, with $s>3/2$, let $u$ be the corresponding unique solution of \eqref{1.1} with lifespan $T>0$. If there exists $t_0\in[0,T)$ such that $u\Big|_{\{t_0\}\times X} \equiv 0$, then $u \equiv 0$ in $[0,T)\times X$.
\end{corollary}
\begin{proof}
We will proceed with the proof for $X=\R$, while the periodic setting is treated analogously. Suppose $t_0\in[0,T)$ is such that $u\big|_{\{t_0\}\times \R} \equiv 0$. Then
$\Vert u(t_0)\Vert^2_{H^1} = 0.$
On the other hand, since the $H ^1(\R)$-norm is conserved, we know that $0=\Vert u(t_0)\Vert_{H^1} = \Vert u(t)\Vert_{H^1}$ for every $t\in[0,T)$. From the Sobolev Embedding Theorem \cite[page 317]{taylor}, we conclude that $\Vert u(t)\Vert_{L^{\infty}} = 0$, which is the same as saying that $u\equiv 0$ in $[0,T)\times \R$.
\end{proof}

\section{Proof of Theorem \ref{T1.1}}\label{sec4}

From now on, suppose that $f$ is differentiable, does not change sign and $f(v)=0$ if and only if $v=0$. The two main ingredients for the proof of Theorem \ref{T1.1} is the local well-posedness established by Theorem \ref{teo2.2} and conservation of the $H^1$ norm given in Theorem \ref{teo3.2}. This is due to the fact that the former guarantees the uniqueness of solution, while the conservation law allows to extend the result for any given time while dealing only with a localization.

Suppose there exists an open set $\Omega \subset [0,T)\times X$ such that $u(t,x) = 0$ for every $(t,x)\in\Omega$. Then there exist $t_0\in[0,T)$ and real numbers $a<b$ such that 
\begin{align}\label{E2.4}
\{t_0\}\times[a,b]\subset \Omega.
\end{align}
Observe that if $\{t_0\}\times X\subset \Omega$, then Corollary \ref{cor3.2} tells that $u\equiv 0$ in $[0,T)\times X$ and Theorem \ref{T1.1} is proved. Conversely, assume that $\{t_0\}\times X\subsetneqq \Omega$ and we will now show that the same result holds, finishing the proof of Theorem \ref{T1.1}.

Our strategy is essentially geometric and can be summarised as the follows:
\begin{itemize}
    \item {\bf Real case}: We firstly show that if the conditions in Theorem \ref{T1.1} hold, then we can find a vertical straight line segment \eqref{E2.4} contained in $[0,T)\times\R$ such that the function $x\mapsto (1-\partial_x^2)^{-1}\partial_x f(u(t_0,x))$ vanishes. Then we use the non-locality of the convolution to show that the function $f(u)$ vanishes on the whole line containing such segment;
    \item {\bf Periodic case}: Similarly to the real case, here can identify \eqref{E2.4} with an arc contained in a circle and we will show that the function $x\mapsto (1-\partial_x^2)^{-1}\partial_x f(u(t_0,x))$ vanishes there and, again because of the convolution, this property holds for the whole circle containing that arc.
\end{itemize}

Our goal is then achieved invoking the conservation of the Sobolev norm of the solutions to extend the result to the whole domain of the solution.

\begin{proposition}\label{P2.2} Let $u_0\in H^4(X)$ be a solution of \eqref{1.1} and $u$ its corresponding solution, given by Theorem \ref{2.2}. 
For $t_0, a $ and $b$ given as in \eqref{E2.4} and $F(x) := (1-\partial_x^2)^{-1}\partial_x f(u)(t_0,x)$, we have $F\Big|_{[a,b]}\equiv 0$ on the solutions of \eqref{1.1}.
\end{proposition}
\begin{proof}

From \eqref{2.1} we have $(1-\partial_x^2)^{-1}\partial_x f(u) = -u_t$. 
Since $u$ vanishes on $\Omega$, evaluating the expression for the BBM equation for $t=t_0$ yields
$$F(x) = (1-\partial_x^2)^{-1}\partial_x f(u)(t_0,x) = -u_t(t_0,x) = 0.$$
Therefore, $F(x)=0$ for every $x\in[a,b]$.
\end{proof}

Observe now that $(\Lambda^{-2}f)(x)=g\ast f$, where $g$ is given by \eqref{2.2}.

Since $f$ does not change sign, we know that $(\Lambda^{-2}f)(x)$ also does not change sign. We might assume that $(\Lambda^{-2}f)(x)\geq 0$ and the non-positive case is treated similarly. Moreover, we have that $\Lambda^{-2}f =0$ if and only if $f=0$. Under the hypothesis that $f(x)=0$ if and only if $x=0$, it is then enough to show that $f=0$ to make use of Corollary \ref{cor3.2} to conclude that $u=0$.

From the identity $\partial_x^2\Lambda^{-2} = \Lambda^{-2}-1$ and the fact that $u(t_0,x)=0$ for $x\in [a,b],$ we use Proposition \ref{P2.2} obtain
\begin{align*}
    0 = F(b)-F(a) = \int_a^b F'(x)dx = \int_a^b(\Lambda^{-2}f)(u)(t_0,x)dx.
\end{align*}
Therefore, we have that $f(u)(t_0,x) = 0$, which means that $u(t_0,x)=0$ for every $x \in X$. Because of the conservation of the $H^1$ norm, we have $\Vert u(t)\Vert_{H^1} = \Vert u(t_0)\Vert_{H^1}=0$. Similarly to what was done previously, we conclude that $\Vert u\Vert_{L^{\infty}} = 0$ and $u\equiv 0$ on $[0,T)\times X$.

\section{Discussion and conclusion}\label{sec5}

In this work we give a geometrical and simple proof that if a solution of \eqref{1.1} vanishes on an open set of its domain, where $f$ satisfies the conditions in Theorem \ref{T1.1}, then the solution vanishes everywhere.

A key result to proof Theorem \ref{T1.1} is Proposition \ref{P2.2}, which is based on \cite[Theorem 1.6]{linares}. Proposition \ref{P2.2} shows that we can construct a vertical straight line (in the real case) or a circle (in the periodic case) such that the solution of \eqref{1.1} vanishes. Then, as observed in \cite{igor}, we use the invariance of the $H^1(X)$ norm to show that the solution vanishes on its domain.

Finally, we observe that in \cite[Theorem 4.1]{rosier} the authors presented a unique continuation result for \eqref{1.1} with non-negative $f(u)$ for the periodic case. We do not know an analogous result to the real case. In this communication, not only we present a new demonstration for \cite[Theorem 4.1]{rosier}, but also we prove an analogous result to the real case.

\end{document}